\documentclass[twoside, a4paper, 10pt]{amsart}
\title[ ]{Reconstructing a minimal topological dynamical system from a set of return times}
\usepackage{amsfonts}
\usepackage{amsthm}
\usepackage{verbatim}
\usepackage{amsmath, amssymb}
\usepackage{tikz}
\usetikzlibrary{matrix, arrows}
\usepackage{listings}
\usepackage{color}
\usepackage{listings}
\usepackage[all]{xy}
\usepackage[pdftex,colorlinks,linkcolor=blue,citecolor=blue]{hyperref}
\usepackage{graphicx}
\usepackage{float}
\usepackage[margin=3cm]{geometry}
\usepackage{bigints}
\usepackage{dsfont}
\author{Kamil Bulinski}
\address{School of Mathematics and Statistics, University of Sydney, Australia}
\email{kamil.bulinski@sydney.edu.au}

\author{Alexander Fish}
\address{School of Mathematics and Statistics, University of Sydney, Australia}
\email{alexander.fish@sydney.edu.au}

\begin{document}
\maketitle
\raggedbottom

\newcommand{\cA}{\mathcal{A}}
\newcommand{\cB}{\mathcal{B}}
\newcommand{\cC}{\mathcal{C}}
\newcommand{\cD}{\mathcal{D}}
\newcommand{\cE}{\mathcal{E}}
\newcommand{\cF}{\mathcal{F}}
\newcommand{\cG}{\mathcal{G}}
\newcommand{\cH}{\mathcal{H}}
\newcommand{\cI}{\mathcal{I}}
\newcommand{\cJ}{\mathcal{J}}
\newcommand{\cK}{\mathcal{K}}
\newcommand{\cL}{\mathcal{L}}
\newcommand{\cM}{\mathcal{M}}
\newcommand{\cN}{\mathcal{N}}
\newcommand{\cO}{\mathcal{O}}
\newcommand{\cP}{\mathcal{P}}
\newcommand{\cQ}{\mathcal{Q}}
\newcommand{\cR}{\mathcal{R}}
\newcommand{\cS}{\mathcal{S}}
\newcommand{\cT}{\mathcal{T}}
\newcommand{\cU}{\mathcal{U}}
\newcommand{\cV}{\mathcal{V}}
\newcommand{\cW}{\mathcal{W}}
\newcommand{\cX}{\mathcal{X}}
\newcommand{\cY}{\mathcal{Y}}
\newcommand{\cZ}{\mathcal{Z}}
\newcommand{\bA}{\mathbb{A}}
\newcommand{\bB}{\mathbb{B}}
\newcommand{\bC}{\mathbb{C}}
\newcommand{\bD}{\mathbb{D}}
\newcommand{\bE}{\mathbb{E}}
\newcommand{\bF}{\mathbb{F}}
\newcommand{\bG}{\mathbb{G}}
\newcommand{\bH}{\mathbb{H}}
\newcommand{\bI}{\mathbb{I}}
\newcommand{\bJ}{\mathbb{J}}
\newcommand{\bK}{\mathbb{K}}
\newcommand{\bL}{\mathbb{L}}
\newcommand{\bM}{\mathbb{M}}
\newcommand{\bN}{\mathbb{N}}
\newcommand{\bO}{\mathbb{O}}
\newcommand{\bP}{\mathbb{P}}
\newcommand{\bQ}{\mathbb{Q}}
\newcommand{\bR}{\mathbb{R}}
\newcommand{\bS}{\mathbb{S}}
\newcommand{\bT}{\mathbb{T}}
\newcommand{\bU}{\mathbb{U}}
\newcommand{\bV}{\mathbb{V}}
\newcommand{\bW}{\mathbb{W}}
\newcommand{\bX}{\mathbb{X}}
\newcommand{\bY}{\mathbb{Y}}
\newcommand{\bZ}{\mathbb{Z}}

\newcounter{dummy} \numberwithin{dummy}{section}

\theoremstyle{definition}
\newtheorem{mydef}[dummy]{Definition}
\newtheorem{prop}[dummy]{Proposition}
\newtheorem{corol}[dummy]{Corollary}
\newtheorem{thm}[dummy]{Theorem}
\newtheorem{lemma}[dummy]{Lemma}
\newtheorem{eg}[dummy]{Example}
\newtheorem{notation}[dummy]{Notation}
\newtheorem{remark}[dummy]{Remark}
\newtheorem{claim}[dummy]{Claim}
\newtheorem{Exercise}[dummy]{Exercise}
\newtheorem{question}[dummy]{Question}

\begin{abstract} We investigate to what extent a minimal topological dynamical system is uniquely determined by a set of return times to some open set. We show that in many situations this is indeed the case as long as the closure of this open set has no non-trivial translational symmetries. For instance, we show that under this assumption two Kronecker systems with the same set of return times must be isomorphic. More generally, we show that if a minimal dynamical system has a set of return times that coincides with a set of return times to some open set in a Kronecker system with translationarily asymmetric closure, then that Kronecker system must be a factor. We also study similar problems involving Nilsystems and polynomial return times. We state a number of questions on whether these results extend to other homogeneous spaces and transitive group actions, some of which are already interesting for finite groups.

\end{abstract}

\section{Introduction}

We study the question of whether knowing the exact times a point in an unknown dynamical system enters an open set is enough to determine the system. In this paper, a topological dynamical system is a pair $(X,T)$ where $X$ is a compact metric space with metric $d$ (all metrics will be refered to as $d$ even for different spaces if no ambiguity arises) and $T:X \to X$ is a homeomorphism. Recall that $(X,T)$ is $\textit{minimal}$ if for all $x_0 \in X$ the orbit $T^{\bZ}x_0 = \{T^n x_0 ~|~ n \in \bZ\}$ is dense and is \textit{transitive} if some orbit is dense. If $x_0 \in X$ and $U \subset X$ is an open set then we can define the set of return times $$\mathcal{R}_{(X,T)}(x_0, U) = \{ n \in \bZ ~|~ T^nx_0 \in U \}.$$ Thus our question may now be stated as follows: Given two minimal topological dynamical systems  $(X_1, T_1)$ and $(X_2, T_2)$ with $x_1 \in X_1, x_2 \in X_2$ and $U_1 \subset X_1, U_2 \subset X_2$ open such that $\mathcal{R}_{(X_1,T_1)}(x_1, U_1) = \mathcal{R}_{(X_2,T_2)}(x_2, U_2)$, then is it necessarily true that $(X_1, T_1)$ and $(X_2, T_2)$ are isomorphic? If so, is there an isomorphism mapping $x_1$ to $x_2$?

Recall that an isomorphism $\phi:(X_1, T_1) \to (X_2, T_2)$ is a homeomorphism $\phi:X_1 \to X_2$ such that $\phi \circ T_1 = T_2 \circ \phi$.  One can contrast this to the Taken's Reconstruction Theorem \cite{TakensOriginal} as well as recent developments \cite{GutmanTakens} where one instead is given a system $(X,T)$ where the $\textit{dimension}$ of $X$ is $d$ and asks whether the delay observation mapping $X \to [0,1]^{2d+1}$ given by $x \mapsto (h(x), h(Tx), \ldots, h(T^{2d+1}x))$ is injective for some \textit{generic} continuous $h:X \to [0,1]$. Thus our question can instead be posed as asking to what extent the mapping $$(X, T, x_0, U) \mapsto \left(\mathds{1}_U(T^n x_0) \right)_{n \in \bZ}$$ is injective up to isomorphism.

\subsection{Reconstructing Kronecker Systems}
Let us now start with some motivating examples where no such isomorphism exists.

\begin{eg} Let $X_1 = X_2 = \bT = \bR/\bZ$ and suppose $\alpha \in \bT$ is irrational. Let $$U_1=U_2 = (-\epsilon, \epsilon) \cup (-\epsilon + \frac{1}{2}, \epsilon + \frac{1}{2}) \subset \bT.$$ Then for each $n \in \bZ$ we have that $n\alpha \in U_1$ if and only if $n(\alpha + \frac{1}{2}) \in U_2$. However, the minimal systems $(\bT, x \mapsto x+\alpha)$ and $(\bT, x \mapsto x + \alpha + \frac{1}{2})$ are not isomorphic.

\end{eg}

\begin{eg}\label{eg: Torus factor} Let $X_1 = \bT$, $X_2 = \bT^2$, $\alpha_1 = \sqrt{2} \in \bT$ and $\alpha_2 = (\sqrt{2}, \sqrt{3}) \in \bT^2$. Let $U_1 \subset \bT$ be any non-empty proper open subset and let $U_2 = U_1 \times \bT$. Then clearly $n \alpha_1 \in U_1$ if and only if $n\alpha_2 \in U_2$. Yet $X_1$ and $X_2$ are not even homeomorphic.

\end{eg}

These two examples highlight that open sets with non-trivial translational symmetries are a source of issues. For a compact abelian metrizable group $(K, +)$ and $A \subset K$ we can define the stabilizer $$\operatorname{Stab}_K(A) = \{k \in K ~|~ A + k = A \}.$$ Note that in our examples above the stabilizers are not trivial; the stabilizer of $(-\epsilon, \epsilon) \cup (-\epsilon + \frac{1}{2}, \epsilon + \frac{1}{2}) $ is $\{0, \frac{1}{2}\}$ while the stabilizer of $U_1 \times \bT$, for $U_1 \subset \bT$ non-empty and proper, is the vertical subgroup $\{0\} \times \bT$.

Our first main result demonstrates that the only way two non-isomorphic Kronecker systems can yield the same set of return times is if the closure of one of the defining open sets has a non-trivial stabilizer.

\begin{thm}\label{thm: Kronecker systems isomorphic} Let $(K_1, +)$ and $(K_2, +)$ be two compact metrizable abelian groups and suppose that $\alpha_1 \in K_1$ and $\alpha_2 \in K_2$ are such that $\overline{\bZ\alpha_i} = K_i$ for $i=1,2$. Suppose that $U_1 \subset K_1$ and $U_2 \subset K_2$ are open sets such that the stabilizers of their closures are trivial, i.e., $\operatorname{Stab}_{K_i}(\overline{U_i}) = \{0\}$. Then if $$\{n \in \bZ ~|~ n\alpha_1 \in U_1\} = \{n \in \bZ ~|~ n\alpha_2 \in U_2\}$$ then there exists an isomorphism of topological groups (so continuous) $K_1 \to K_2$ mapping $\alpha_1$ to $\alpha_2$.\end{thm}

Note that if one removes a point not on the trajectory $\{n\alpha_i ~|~ n \in \bZ\}$ from $U_i$ then $U_i$ remains open and the return times and the closures do not change, so it is indeed natural to study the stabilizer of the closure $\overline{U_i}$ rather than $U_i$.

\subsection{Detecting Kronecker Factors of Minimal Systems}

We now turn our attention to minimal topological systems. As demonstrated in Example~\ref{eg: Torus factor} if a minimal system has a non-trivial factor, then it shares a set of return times with it. Of course, this is also the case if one replaces in that example the rotation $(x,y) \mapsto (x+\sqrt{2},y+\sqrt{3})$ with an ergodic skew product such as $\bT^2 \to \bT^2: (x,y) \mapsto (x+\sqrt{2}, x+y)$. Our next result shows that essentially the only way a minimal system can share a set of return times with a Kronecker system is if the latter is a factor of the former.

\begin{thm}\label{thm: Kronecker factor of minimal system} Let $(X,T)$ be a minimal topological dynamical system and $U \subset X$ be a non-empty open set and $x_0 \in X$. Let $(K, +)$ be a compact metrizable abelian group and $\alpha \in K$ such that $\overline{\bZ \alpha} = K$ and let $U' \subset K$ be a non-empty open set. Suppose that \begin{align}\label{equality of return times} \{ n \in \bZ ~|~ T^nx_0 \in U \} = \{n \in \bZ ~|~ n\alpha \in U'\} \end{align} and that $\operatorname{Stab}_K(\overline{U'}) = \{ 0 \}$. Then the pointed Kronecker system $(K,0 , k \mapsto k+\alpha)$ is a factor of the pointed system $(X, x_0, T)$, i.e., there is a continuous map $\phi: X \to K$ with $\phi(x_0) = 0$ satisfying $$\phi(Tx) = \phi(x) + \alpha \text{ for all } x \in X.$$

\end{thm}

Note that Theorem~\ref{thm: Kronecker factor of minimal system} immediately implies Theorem~\ref{thm: Kronecker systems isomorphic} (apply twice to both Kronecker systems to obtain two factor maps that are inverses) since an isomorphism of Kronecker systems that preserves the zero elements is necessarily a group isomorphism. The following example shows that we cannot replace the minimality assumption with transitivity.

\begin{eg}(Insufficient to assume transitivity instead of minimality) Let $X \subset [-1, 1]$ be the set $X = \{ x_n ~|~ n \in \bZ \cup \{-\infty, \infty \} \}$ where $x_{\pm \infty} = \pm 1$ and for $n\in \bZ$,  

$$x_n =  \begin{cases}
    -1 + \frac{1}{|n|} & n < 0 \\
    0 & n = 0 \\
    1 - \frac{1}{n} & n >0
  \end{cases}$$

Thus $X$ is closed and all $x_n$ with $n \in \bZ$ are isolated points. So we have a homeomorphism $T:X \to X$ mapping $x_n$ to $x_{n+1}$ and fixing $x_{\pm \infty}$. Hence it is transitive (all $x_n$ with $n \in \bZ$ have dense orbits) but not minimal. Now $U = \{x_n ~|~ n \in 2\bZ\}$ is open and the return times set $2\bZ = \{ n \in \bZ ~|~ T^n x_0 \in U\}$ is also equal to $\{n \in \bZ ~|~ n\alpha \in U'\}$ where $U' = \{0\} \subset K = \bZ/2\bZ$ is open and $\alpha = 1 \in K$. But there is no factor $\pi$ from $(X,T)$ to the Kronecker system $(K, k \mapsto k+1)$ as it would yield the contradiction $\pi(x_{\infty}) = \pi(Tx_{\infty}) = \pi(x_{\infty}) + 1$.

\end{eg}

\subsection{Return times along a polynomial sequence}

We recall the following result which follows immediately from the Polynomial Weyl Equidistribution \cite{WeylPolynomial} and the fact that connected compact abelian groups have no non-trivial characters with finite image.

\begin{prop} Let $K$ be a \textbf{connected} compact abelian group and suppose that $P(x) \in \bZ[x]$ is a non-constant polynomial. If $\alpha \in K$ is such that $\overline{\bZ \alpha} = K$ then $P(\bZ)\alpha = \{P(n)\alpha ~|~ n \in \bZ\}$ is dense in $K$.  

\end{prop}

Given this, it is natural to ask whether we can extend Theorem~\ref{thm: Kronecker systems isomorphic} to the case where we only know the set of return times inside a polynomial sequence, i.e., the set $\{ n ~|~ P(n)\alpha \in U\}$. The next result shows that this is indeed the case.

\begin{thm}\label{thm: polynomial return times} Let $K_1$ and $K_2$ be compact connected abelian groups and suppose that $\alpha_i \in K_i$ with $\overline{\bZ\alpha_i} = K_i$. Let $P(x) \in \bZ[x]$ be a polynomial with $P(0) = 0$. Let $U_i \subset K_i$ be non-empty open sets and let $\mathcal{R}_i = \{n \in \bZ ~|~ P(n)\alpha_i \in U_i\}$. If $\mathcal{R}_1 = \mathcal{R}_2$ and $\operatorname{Stab}_{K_1}(\overline{U_1}) = \{0 \}$ and $\operatorname{Stab}_{K_2}(\overline{U_2}) = \{0 \}$ then there exists an isomorphism $K_1 \to K_2$ (of topological groups, so continuous) mapping $P(n)\alpha_1$ to $P(n)\alpha_2$ for each $n \in \bZ$. In particular, if $P(\bZ)$ is not contained in any proper subgroup of $\bZ$ then this isomorphism maps $\alpha_1$ to $\alpha_2$. \end{thm}

\begin{eg} Consider the polynomial $P(n) = n^5-n$ and let $\alpha_1 = \sqrt{2} \in \bT$ and $\alpha_2 = \sqrt{2} + \frac{1}{5} \in \bT$. Then for any open set $U \subset \bT$ we have that $\{n \in \bZ ~|~ P(n)\alpha_1 \in U\} =  \{n \in \bZ ~|~ P(n)\alpha_2 \in U\} $ and indeed there is an isomorphism (the identity map) mapping $P(n)\alpha_1 $ to $P(n)\alpha_2$ as they are equal, but there is no isomorphism mapping $\alpha_1$ to $\alpha_2$. 

\end{eg}

We remark that it seems that there is no obvious way to extend Theorem~\ref{thm: Kronecker factor of minimal system} to polynomial times under the assumption that $X$ is connected since an example of Pavlov \cite{PavlovCounterexamples} shows that there exists a minimal connected system $(X,T)$ such that $\{T^{n^2} x_0 ~|~ n \in \bZ\}$ is not dense for some $x_0 \in X$.

\subsection{Return times sets of nilmanifolds}

It is natural to try to extend Theorem~\ref{thm: Kronecker systems isomorphic} on the return times in Kronecker systems to other homogeneous spaces such as nilmanifolds. The following is a partial result on this which establishes that if a Kronecker system $(K,+)$ and a minimal nilsystem $(G/\Gamma, T)$ share the same return times, with the closures of the defining open sets having trivial stabilizers under translations of $K$ and $G$ respectively, then they must be isomorphic as dynamical systems.

\begin{thm}\label{thm: nilsystem and kronecker} Let $X=G/\Gamma$ be a nilmanifold, where $G$ is a nilpotent Lie group and $\Gamma \leq G$ is a cocompact discrete subgroup. Suppose that $\tau \in G$ is such that $(X,T)$ is minimal, where $T:X \to X$ is given by $Tx = \tau x$ for $x \in X$. Suppose that $U \subset X$ is open and satisfies that $\operatorname{Stab}_G(\overline{U}):= \{g \in G ~|~ g\overline{U} = \overline{U}\}$ consists of only those $g \in G$ such that $gx = x$ for all $x \in X$. Now suppose that $(K,+)$ is a compact abelian group and $\alpha \in K$ is such that $\overline{\bZ \alpha} = K$ and $V \subset K$ is open such that $\operatorname{Stab}_K(\overline{V}) = \{0\}$ and $$\{ n \in \bZ ~|~ T^nx_0 \in U\} = \{n \in \bZ ~|~ n\alpha \in V\}$$ for some $x_0 \in X$. Then the pointed system $(X, x_0, T)$ is isomorphic to the pointed Kronecker system $(K,0 , k \mapsto k+\alpha)$, i.e., there is a homeomorphism $\phi: X \to K$ with $\phi(x_0) = 0$ satisfying $$\phi(Tx) = \alpha + \phi(x) \text{ for all } x \in X.$$

\end{thm}

Note that Theorem~\ref{thm: Kronecker factor of minimal system} shows that such a continuous map $\phi$ exists but does not show that it is a homeomorphism (and indeed it may not be without the assumption on $\operatorname{Stab}_G(\overline{U})$). Thus it suffices to show the existence of an inverse to $\phi$, which will be established in Section~\ref{sec: Nilsystems}.

It is interesting to ask whether the same result holds more generally for two nilsystems rather than a nilsystem and a Kronecker system. Furthermore, one can extend this question to more general homogeneous spaces.

\subsection{An explicit reconstruction for Jordan Measurable sets} One way of stating Theorem~\ref{thm: Kronecker factor of minimal system} is that one can uniquely reconstruct a Kronecker system $(K, x \mapsto x+\alpha)$, where $K$ is a compact group with $\overline{\bZ\alpha} = K$, from the set $\mathcal{R} = \{n ~|~ n\alpha \in U\}$ provided that $U$ is some (unknown) open set with $\operatorname{Stab}_K(\overline{U})$ is trivial. If we further assume that $U$ is Jordan measurable (recall that this means that $m_K(\overline{U}) = m_K(U)$ where $m_K$ is the Haar measure) then the following result shows that there is in fact a rather explicit reconstruction of $(K, \alpha)$ from $\mathcal{R}$. Let $\bU = \{ z \in \bC ~|~ |z|=1\}$ denote the unit complex numbers.

\begin{thm}\label{thm: explicit reconstruction} Let $K$ be a compact abelian group and let $\alpha \in K$ be such that $\overline{\bZ\alpha} = K$ and suppose $U \subset K$ is a Jordan measurable open set such that $\operatorname{Stab}_K(\overline{U})$ is trivial. Let $\mathcal{R} = \{n ~|~ n\alpha \in U\}$ and let $$ \Lambda = 
\left\{ \lambda \in \bU ~|~ \frac{1}{N}\sum_{n=1}^N \lambda^n \mathds{1}_{\mathcal{R}}(n) \text{ does not converge to } 0 \text{ as } N \to \infty \right\}.$$ Then $\Lambda = \{ \lambda_1, \lambda_2, \ldots\}$ is countable and there is an injective continuous group homomorphism $K \to \bU^{\bN}$ mapping $\alpha$ to $ \vec{\lambda} = (\lambda_1, \lambda_2, \ldots )$. In particular, $K$ is isomorphic to the closure of the subgroup generated by $\vec{\lambda}$. \end{thm}

Note that this implies Theorem~\ref{thm: Kronecker factor of minimal system} in the case where the open sets are Jordan measurable but it is not difficult to construct examples with $U$ not Jordan measurable where this reconstruction is invalid despite Theorem~\ref{thm: Kronecker factor of minimal system} still guaranteeing the uniqueness of the Kronecker system.

\subsection{Further questions}

We now gather some open questions. Our first question asks about a natural extension of Theorem~\ref{thm: nilsystem and kronecker} to two nilsystems and more general homogeneous spaces.

\begin{question} Let $X_1=G_1/\Gamma_1$ and $X_2=G_2/\Gamma_2$ be homogeneous spaces, where $G_i$ are Polish groups and $\Gamma_i \leq G_i$ are cocompact discrete subgroups. For all $i=1,2$, suppose that $U_i \subset X_i$ is an open set with closure having a trivial stabilizer in the sense that whenever $g\in G$ is such that $g\overline{U_i} = \overline{U_i}$ then $gx=x$ for all $x \in X_i$. Suppose that $\tau_i \in G_i$ are such that $(X_i, x \mapsto \tau_i x)$ are minimal systems and $$\{n \in \bZ ~|~ \tau_1^n x_1 \in U_1\} = \{n \in \bZ ~|~ \tau_2^n x_2 \in U_2\}$$ for some $x_1 \in X_1$ and $x_2 \in X_2$. Then are the systems  $(X_1, x \mapsto \tau_1 x)$ and $(X_2, x \mapsto \tau_2 x)$ isomorphic (via a map sending $x_1$ to $x_2$)?  If not true in general, is it true if $G_i$ are nilpotent Lie groups?\end{question}

The next question asks whether one can extend our results on minimal dynamical $\bZ$ systems to actions of other groups. It turns out that this seems difficult even without any topology, so we assume full orbits rather than just dense orbits as follows.

\begin{question}\label{question: G sets} Suppose that $G$ is a group acting on sets $X_1$ and $X_2$ transitively. Thus there are $x_1 \in X_1$ and $x_2 \in X_2$ with $Gx_1 = X_1$ and $Gx_2 = X_2$. Suppose that $U_1 \subset X_1$ and $U_2 \subset X_2$ are subsets with trivial setwise stabilizer, i.e., $$\{g \in G ~|~ gU_1 = U_1\} = \{1\} = \{g \in G ~|~ gU_2 = U_2\}.$$ If $\{g \in G ~|~ gx_1 \in U_1\} = \{g \in G ~|~ gx_2 \in U_2\}$ then does it mean that the two actions are isomorphic? That is, does there exist a bijection $\phi:X_1 \to X_2$ such that $\phi(gx) = g\phi(x)$ for all $g \in G$ and $x \in X_1$? Does there exist one mapping $x_1$ to $x_2$? \end{question}

Using the language of \cite{DistinguishableOriginal} and \cite{DistinguishabilityInfiniteGroups} such group actions satisfying the trivial setwise stabilizer hypothesis are called $2$-distinguishable. This question is already interesting for $G$ finite. In fact, we can simplify this to the following interesting question, which again is already interesting for finite sets (we will show this equivalence in Section~\ref{section: equivalence of questions}).

\begin{mydef} Let $G$ be a group acting on a set $X$. We say that $U \subset X$ is \textit{simple} (for this action) if whenever $\pi:X \to X'$ is a factor (i.e., $X'$ is a $G$-set and $\pi:X \to X'$ is a surjective map such that $\pi(gx) = g\pi(x)$ for all $g \in G$, $x \in X$) such that $U = \pi^{-1}(U')$ for some $U' \subset X'$, then $\pi$ is a bijection (isomorphism of $G$-sets). In other words, $U$ is simple if it can not be realised as a preimage of a factor in a non-trivial way (i.e., the factor is not an isomorphism).\end{mydef}

\begin{question}\label{question: stable implies simple} Let $G$ be a group acting transitively on a set $X$ and suppose that $U \subset X$ has trivial setwise stabilizer, i.e., $\{g \in G ~|~ gU = U\}$ consists of only $1 \in G$. Then is it true that $U$ is simple for this $G$-action? In other words, is it true that if $\mathcal{B} \subset 2^X$ is a $G$-invariant partition of $X$ such that $U$ is a union of elements in $\mathcal{B}$, then $\mathcal{B}$ consists of only singletons?

\end{question}

\textbf{Acknowledgements:} The authors were partially supported by the Australian Research Council grant DP210100162. They are grateful to Sean Gasiorek and Robert Marangell for interesting discussions.

\section{Bohr Set Return Times of Minimal Systems (Proof of Theorem~\ref{thm: Kronecker factor of minimal system})}

\begin{proof}[Proof of Theorem~\ref{thm: Kronecker factor of minimal system}] Consider the product system $(Z, S)$ where $Z = X \times K$ and $S:Z \to Z$ is given by $$S(x,k) = (Tx, k+\alpha).$$ Let $\Theta = \overline{S^{\bZ}z_0}$ be the orbit closure of $z_0 = (x_0, 0) \in Z$. Let $\pi_X:\Theta \to X$ and $\pi_K: \Theta \to K$ denote the projection maps, which are surjective by the minimality of $(X,T)$ and $(K, k \mapsto k+\alpha)$. Let $$H = \{ h \in K ~|~ (x_0, h) \in \Theta\}.$$ 

\textbf{Claim:} $H$ is a closed subgroup of $K$. \\
Note that $(x_0, 0) = z_0 \in \Theta$ so $0 \in H$. Now if $h, h' \in H$ then that means $(x_0, h), (x_0, h') \in \Theta$ and thus there exists a sequence $n_1, n_2, \ldots \in \bZ$ such that $$S^{n_i}z_0 = (T^{n_i} x_0, n_i \alpha) \to (x_0, h)$$ as $i \to \infty$ and another sequence $n'_1, n'_2 \ldots \in \bZ$ such that $(T^{n'_i} x_0, n'_i \alpha) \to (x_0, h')$. Now fix $\epsilon > 0.$ We may choose $J = J(\epsilon) \in \bN$ such that\footnote{By this we mean that we have fixed an invariant metric $d_K$ on $K$ and we write $k \approx_{\epsilon} k'$ if $d_K(k, k') < \epsilon$ for $k,k' \in K$. } $n_J\alpha \approx_{\epsilon} h$ and choose $\delta >0$ such that $$d(T^{n_J}x , x_0) < \epsilon \text{ whenever } d(x,x_0) < \delta.$$ Now take $I$ large enough so that $d(T^{n'_I}x_0, x_0) < \delta$ and $n'_I \alpha \approx_{\epsilon} h'$. Thus $d(T^{n_J + n'_I}x_0, x_0) < \epsilon$ and $(n_J + n'_I)\alpha \approx_{2\epsilon} h + h'$. As $\epsilon>0$ is arbitrary, this means $(x_0, h + h') \in \overline{S^{\bZ}z_0} = \Theta$, thus $h+h' \in H$. Finally, notice that $H$ is closed (as it is $\pi_K(\Theta \cap (\{x_0\} \times K))$) and so it is a closed non-empty sub-semigroup and thus a closed subgroup as $K$ is compact. 

Now for $x \in X$ let $H_{x} = \{k \in K ~|~ (x, k) \in \Theta \}$, thus $H = H_{x_0}$. 

\textbf{Claim:} For each $x\in X$, we have that $H_{x}$ is a coset of $H$, i.e., $H_x$ and $H_{x'}$ are translates of each other for all $x, x' \in X$. \\
To see this note that by minimality of $(X,T)$, there exists a sequence $n_1, n_2, \ldots \in \bZ$ such that $T^{n_i}x \to x'$ and by compactness we suppose that $n_i\alpha \to \alpha_{x,x'} \in K$. It follows that $H_{x} + \alpha_{x,x'} \subset H_{x'}$. In particular this means that $\alpha_{0, x} + \alpha_{x, 0} \in H$, so $$H + \alpha_{0,x} = H - (\alpha_{0, x} + \alpha_{x, 0}) + \alpha_{0,x} = H - \alpha_{x,0} \supset H_x$$ and thus $H_x = H + \alpha_{0,x}$ is a coset of $H$ as desired.

\textbf{Claim:} $H \subset \operatorname{Stab}_K(\overline{U'})$. \\
To see this suppose that $h \in H$ and $u' \in U'$. Then by the surjectivity of $\pi_K:\Theta \to K$ there exists $x \in X$ such that $(x, u') \in \Theta$. Now we may find a sequence $n_1, n_2, \ldots \in \bZ$ such that $n_i \alpha \to u'$ and $T^{n_i}x_0 \to x$. So for large enough $i$, we have $n_i\alpha \in U'$ and thus by the assumption of the equality of return times (\ref{equality of return times}) we have $T^{n_i}x_0 \in U$. Since $(T^{n_i}x_0, n_i\alpha) \in \Theta$ we must have that $(T^{n_i}x_0, n_i\alpha + h) \in \Theta$ as $H_{T^{n_i}x_0}$ is a coset of $H$. This means that, for each fixed $i$, we may find a sequence $m_1, m_2, \ldots \in \bZ$ such that $S^{m_j}z_0 \to (T^{n_i}x_0, n_i\alpha + h)$. So for large enough $j$ we have that $T^{m_j}x_0$ is so close to $T^{n_i}x_0 \in U$ that $T^{m_j}x_0 \in U$ and so again by the equality of return times we have $m_j\alpha \in U'$. But as $m_j \alpha \to n_i\alpha + h$, we must have that $n_i\alpha + h \in \overline{U'}$. Finally, as $n_i\alpha \to u'$ we must have that $u' + h \in \overline{U'}$. Thus as $u' \in U'$ is arbitrary we must have that $h + U' \subset \overline{U'}$ and by continuity of translation $h+\overline{U'} \subset \overline{U'}$. Thus $h \in \operatorname{Stab}_K(\overline{U'})$ as claimed. 

Now by the assumption that this stabilizer is trivial, we get that $H$ is trivial and thus $\pi_X:\Theta \to X$ is injective and hence a homeomorphism. Thus our desired factor is $\pi_K \circ \pi_X^{-1}:X \to K$. \end{proof}

\section{Polynomial return times (proof of Theorem~\ref{thm: polynomial return times})}

\subsection{The case of compact connected abelian Lie groups}

We now embark on proving Theorem~\ref{thm: polynomial return times}. We first prove this for the case where $K_1$ and $K_2$ are compact connected abelian \textbf{Lie} groups, i.e., subgroups of a finite dimensional torus. We leave the more subtle case of non-Lie groups to the next subsection.

\begin{prop}\label{prop: polynomial orbit closure in compact Lie group} Let $K$ be a compact abelian Lie group with $\alpha \in K$ such that $\bZ\alpha$ is dense in $K$. Let $K_0$ be the connected component of the identity, thus $d = |K / K_0|$ is finite (as $K$ is a compact Lie group). Let $P(x) \in \bZ[x]$ be a non-constant polynomial with $P(0) = 0$. Then $$\overline{P(d\bZ) \alpha} = K_0$$ and $$\overline{P(\bZ)\alpha} = \bigcup_{a \in A} \left(K_0 + a \right)$$ for some finite set $A \subset K$.

\end{prop}

\begin{proof} Note that $dk \in K_0$ for all $k \in K$. In particular $\bZ d \alpha \subset K_0$. Let $K'_0 = \overline{\bZ d\alpha}$. Note that $K'_0 \subset K_0$ and it is of finite index in $K$ since $$K = \overline{\bZ\alpha} =  \bigcup_{m = 0}^{d-1} \overline{(d\bZ + m)\alpha} = \bigcup_{m = 0}^{d-1} \left(K'_0 + m\alpha \right).$$ Thus $K'_0$ is of finite index in $K_0$ as well. But as $K_0$ is connected, we must have that $K'_0 = K_0$. Now let $Q(x) = \frac{1}{d}P(dx) \in \bZ[x]$ (as $P(0) = 0$). From Weyl equidistribution and connectedness of $K_0$ we see that the sequence $Q(1)d\alpha, Q(2)d\alpha, Q(3) d\alpha \ldots $ equidistributes in $K_0$. As $Q(n)(d\alpha) = P(dn)\alpha$ this implies the first claim that $$\overline{P(d\bZ) \alpha} = K_0.$$

Now for the second claim, we write $$\overline{P(\bZ)\alpha} = \bigcup_{m=0}^{d-1} \overline{P(d\bZ + m)\alpha} = \bigcup_{m=0}^{d-1} \overline{P_m(d\bZ)\alpha  + a_m} $$ for some non-constant $P_m(x) \in \bZ[x]$ with $P_m(0) = 0$ and $a_m \in K$. Thus applying the previous claim to the $P_m$ we get that $\overline{P_m(d\bZ)\alpha  + a_m}$ is a coset of $K_0$, as required. \end{proof}

\begin{prop} Let $K_1$ and $K_2$ be compact connected abelian Lie groups and suppose that $\alpha_i \in K_i$ with $\overline{\bZ\alpha_i} = K_i$. Let $P(x) \in \bZ[x]$ be a polynomial with $P(0) = 0$ such that $P(\bZ)$ is not contained in any proper subgroup of $\bZ$. Let $U_i \subset K_i$ be non-empty open sets and let $\mathcal{R}_i = \{n \in \bZ ~|~ P(n)\alpha_i \in U_i\}$. If $\mathcal{R}_1 = \mathcal{R}_2$ and $\operatorname{Stab}_{K_1}(\overline{U_1}) = \{0 \}$ and $\operatorname{Stab}_{K_2}(\overline{U_2}) = \{0 \}$ then there exists an isomorphism $K_1 \to K_2$ (of topological groups, so continuous) mapping $\alpha_1$ to $\alpha_2$.

\end{prop}

\begin{proof} Let $\alpha = (\alpha_1, \alpha_2) \in K_1 \times K_2$ and let $K = \overline{\bZ\alpha} \subset K_1 \times K_2$. Let $\Theta = \overline{P(\bZ)\alpha}$. We know that $$\Theta = \bigcup_{a \in A} (K_0 + a)$$ for some finite $A \subset K$ where $K_0$ is the identity connected component of $K$. Let $\pi_i : K \to K_i$ be the projection maps. We know that $\pi_i$ is surjective by the assumption that $\overline{\bZ\alpha_i} = K_i$. In fact, since $K_0$ is finite index in $K$, we have that $\pi_i(K_0)$ of finite index in $K_i$ and thus $\pi_i(K_0) = K_i$ by the assumption that $K_i$ is connected. We now wish to show that $\pi_i$ is injective on $\Theta$. To see this, suppose that $\theta, \theta' \in \Theta$ are such that $\pi_1(\theta) = \pi_1(\theta')$. We know that $\theta \in K_0 + a$ and $\theta' \in K_0 + a'$ for some $a,a' \in A$. We now wish to show that $\pi_2(\theta - \theta') \in \operatorname{Stab}_{K_2}(\overline{U_2})$. To do this, suppose that $u_2 \in U_2$.  From $\pi_2(K_0) = K_2$ we get that $\pi_2(K_0 + a') = K_2$, thus there exists $x \in K_1$ such that $(x, u_2) \in K_0 + a'$. Now $$(x, u_2 + \pi_2(\theta - \theta')) = (x, u_2) + \theta - \theta' \in (K_0 + a') + (K_0 + a - a') = K_0 + a,$$ that is $(x, u_2 + \pi_2(\theta - \theta')) \in \Theta$. Now since $(x, u_2) \in \Theta$ this means there exists a sequence $n_1, n_2, \ldots$ of integers such that $P(n_j) \alpha \to (x, u_2)$. In particular, $P(n_j)\alpha_2 \to u_2$  so $P(n_j)\alpha_2 \in U_2$ for sufficiently large $j$. But since $\mathcal{R}_1 = \mathcal{R}_2$ this means that $P(n_j)\alpha_1 \in U_1$. Now since $(x, u_2) \in K_0 + a'$ and $K_0 + a'$ is open in $K$, we must have that $P(n_j)\alpha \in K_0 + a'$ for sufficiently large $j$. Thus $P(n_j)\alpha + (\theta - \theta') \in K_0 + a \subset \Theta$ for sufficiently large $j$. In particular, this means that we may choose an integer $n'_j$ so that $P(n'_j)\alpha$ is so close to $P(n_j)\alpha + (\theta - \theta')$ so that $P(n'_j)\alpha_1$ is sufficiently close to $P(n_j)\alpha_1 \in U_1$ so that $P(n'_j)\alpha_1 \in U_1$ and so that $P(n'_j)\alpha \to (x, u_2 + \pi_2(\theta - \theta'))$. But as $\mathcal{R}_1 = \mathcal{R}_2$ this means that $P(n'_j)\alpha_2 \in U_2$. But $P(n'_j)\alpha_2 \to u_2 +\pi_2(\theta - \theta')$, thus showing that $u_2 +\pi_2(\theta - \theta') \in \overline{U_2}$. So we have shown that if $u_2 \in U_2$ then $u_2 +\pi_2(\theta - \theta') \in \overline{U_2}$, and by continuity of addition this holds more generally for $u_2 \in \overline{U_2}$, thus $\pi_2(\theta - \theta') \in \operatorname{Stab}_{K_2}(\overline{U_2})$. But by the assumption that $\operatorname{Stab}_{K_2}(\overline{U_2}) = \{0\}$ we get that $\pi_2(\theta) = \pi_2(\theta')$, which means that $\theta = \theta'$. Thus we have shown that $\pi_1:\Theta \to K_1$ is injective, and thus a homeomorphism. As $K_1$ is connected this means that $\Theta$ is connected and so in fact $ \Theta = K_0$ (as $P(0)\alpha = (0,0)$). So $\pi_1: K_0 \to K_1$ is in fact an isomorphism of topological groups. By symmetry (using now the assumption that  $\operatorname{Stab}_{K_1}(\overline{U_1}) = \{0 \}$) we get that $\pi_2:K_0 \to K_2$ is an isomorphism of topological groups. Thus we have an isomorphism $K_1 \to K_2$ mapping $P(n)\alpha_1$ to $P(n)\alpha_2$ for all $n \in \bZ$. Finally, from the assumption that $P(\bZ)$ is not contained in any proper subgroup, we have that $1 = \sum_{i} a_i P(n_i)$ for some integers $a_i$ and $n_i$. Thus $\alpha_1 = \sum_{i} a_i P(n_i) \alpha_1$ is mapped to $\sum_{i} a_i P(n_i)\alpha_2 = \alpha_2$ under this isomorphism.  \end{proof}

\subsection{Infinite dimensional polynomial orbit}

We now remove the assumption that our groups $K_1$ and $K_2$ are Lie groups (i.e., embedded in a finite dimensional torus). Thus let $K_1$ and $K_2$ be compact metrizable connected abelian groups. Note that this means they are closed subgroups of the countable dimensional torus $\bT^{\bN}$ and thus we may approximate them by compact Lie groups as follows.

\begin{mydef} Let $(X, d_X)$ and $(Y, d_Y)$ be metric spaces. For $\delta>0$, we say that a map $\phi:X \to Y$ is a $\delta$-almost isometry if for all $x_1,x_2 \in X$ we have that $$|d_X(x_1, x_2) - d_Y(\phi(x_1), \phi(x_2))| < \delta.$$

\end{mydef}

Thus we have a decreasing sequence $\delta_1 > \delta_2> \ldots$ of positive real numbers converging to $0$ and, for each positive integer $D$, surjective continuous group homomorphisms $\phi^1_D: K_1 \to F^1_D$ and $\phi^2_D: K_2 \to F^2_D$ where $F^1_D$ and $F^2_D$ are closed subgroups of $\bT^D$ and $\phi^1_D$ and $\phi^2_D$ are $\delta_D$ almost isometries (we have fixed an invariant metric on each compact metrizable abelian group $K$, which we shall call $d_K$ or $d$ if clear from the context). For concreteness, we construct these maps conveniently as follows. Equip $\bT^{D}$ with the metric $$d(x, y) = \sum_{i=1}^{D} 2^{-i}\| x_i - y_i\|,$$ where $\| v\|$ is the shortest distance from $v \in \bT$ to $0 \in \bT$, and $\bT^{\bN}$ with the metric $$d(x, y) = \sum_{i=1}^{\infty} 2^{-i}\| x_i - y_i\|,$$ that way $\phi_D^1$ and $\phi_D^2$ are (restrictions to $K_1$ and $K_2$ respectively) the projections onto the first $D$ co-ordinates, which are clearly $O(2^{-D})$-almost isometries. The $F_D^1$ and $F_D^2$ are then defined to be the images of these, thus $\phi^1_D: K_1 \to F^1_D$ and $\phi^2_D: K_2 \to F^2_D$ are surjective.

Now let $P(x) \in \bZ[x]$ be a non-constant polynomial such that $P(0)=0$. We now let $$\Theta = \overline{P(\bZ) \alpha} \subset K_1 \times K_2$$ where $\alpha = (\alpha_1, \alpha_2) \in K_1 \times K_2$ are such that $\overline{\bZ\alpha_i} = K_i$ for $i=1,2$. We let $\phi_D:K_1 \times K_2 \to F_D^1 \times F_D^2$ be the map $\phi_D = \phi_D^1 \times \phi_D^2$. We also let $\Theta_D = \phi_D(\Theta) \subset F_D^1 \times F_D^2$. We equip $K_1 \times K_2$ and $F_D^1 \times F_D^2$ with the metrics $d((x_1, y_1), (x_2, y_2)) = \max\{d(x_1,x_2), d(y_1, y_2)\}$, that way $\phi_D$ is also a $O(2^{-D})$-almost isometry. Observe that $$\Theta_D = \overline{P(\bZ) \phi_D(\alpha)}$$ and thus we may apply Proposition~\ref{prop: polynomial orbit closure in compact Lie group} to obtain that \begin{align}\label{Theta_D a union of cosets} \Theta_D = \bigcup_{a \in A_D} (\Gamma_D + a) \end{align} for some finite $A_D \subset F_D^1 \times F_D^2$ and closed connected subgroup $\Gamma_D \leq F_D^1 \times F_D^2$. Note that $\Gamma_D \subset \Theta_D$ as $0 \in \Theta_D$ since $P(0)=0$.

Let $$\Gamma = \bigcap_{D=1}^{\infty} \phi_D^{-1}(\Gamma_D)$$ and notice that this is a closed subgroup.

\begin{lemma}\label{lemma: Gamma subset of Theta} We have $\Gamma \subset \Theta$. \end{lemma}

\begin{proof} Since $\phi_D$ is a $O(2^{-D})$-almost isometry, we have that if $p \in \phi_D^{-1}(\Gamma_D)$ then $\phi_D(p) \in \Gamma_D \subset \phi_D(\Theta)$ and so $d(p, \Theta) < O(2^{-D})$. But $\Theta$ is closed, so the proof is complete by letting $D \to \infty$. \end{proof}

\begin{lemma}\label{lemma: nested Gamma_D} We have that $\phi_D^{-1}(\Gamma_D) \supset \phi_{D+1}^{-1}(\Gamma_{D+1})$.

\end{lemma}

\begin{proof} The natural projection $F_{D+1}^1 \times F_{D+1}^2 \to F_{D}^1 \times F_{D}^2$ maps $\Theta_{D+1}$ to $\Theta_D$, so in particular maps the connected component $\Gamma_{D+1}$ to the connected component $\Gamma_D$.  \end{proof}

We let $\pi_i: K_1 \times K_2 \to K_i$ and $\pi_{i,D}:F_D^1 \times F_D^2 \to F_D^i$ denote the projections.

\begin{lemma}\label{lemma: Gamma maps onto K_i} For $i=1,2$ we have that $\pi_i(\Gamma) = K_i$. \end{lemma}

\begin{proof} Fix $i \in \{1,2\}$. We first show that $\pi_{i,D}(\Gamma_D) = F_D^i$. First notice that $$\pi_{i,D}(\Theta_D) = \overline{P(\bZ)\phi_D(\alpha_i)} = F_D^i$$ where the last equality follows from Weyl equdisitribution, connectendess of $F_{D}^i$ and $\overline{\bZ \phi_D(\alpha_i)} = F_D^i$. But$$\pi_{i,D}(\Theta_D) = \bigcup_{a \in A_D} \left(\pi_{i,D}(\Gamma_D) + \pi_{i,D}(a)\right)$$ is a disjoint union of closed sets, thus by connectedness they must all be the same and equal to $\pi_{i,D}(\Gamma_D)$, so indeed $\pi_{i,D}(\Gamma_D) = F_D^i$. Now let $k \in K_i$. Then by our first claim we have that $\phi_D^i(k) = \pi_{i,D}(\gamma_D)$ for some $\gamma_D \in \Gamma_D$. Now by surjectivity of $\phi_D^i$, we may take $\beta_D \in \phi_D^{-1}(\Gamma_D)$ such that $\phi_D(\beta_D) = \gamma_D$. By compactness, we may pass to a subsequence such that $\beta_{D_j} \to \beta$ for some $\beta \in K_1 \times K_2$. But it now follows from Lemma~\ref{lemma: nested Gamma_D} that $\beta \in \Gamma$. The proof will be complete if we can show that $\pi_{i}(\beta) = k$. To see this, note that $\phi_D^i (\pi_i(\beta)) = \pi_{i,D}(\phi_D(\beta))$ but $\phi_D(\beta)$ can be made arbitrarily close to $\phi_D(\beta_D) = \gamma_D$ for large enough $D$. So $\phi_D^i(\pi_i(\beta))$ can be made arbitrarily close to $\pi_{i,D}(\gamma_D) = \phi_D^i(k)$. It now follows that $\pi_i(\beta)$ is arbitrarily close to $k$ for large enough $D$ (as $\phi_D^i$ is a $O(2^{-D})$-almost isometry), thus they are equal. \end{proof}

\begin{prop}\label{prop: Gamma union of cosets} There exists $A \subset K_1 \times K_2$, with $0 \in A$, such that $$\Theta = \bigcup_{a \in A} (\Gamma + a).$$

\end{prop}

\begin{proof} We already know that $\Gamma \subset \Theta$ by Lemma~\ref{lemma: Gamma subset of Theta}. Thus it remains to show that if $\gamma_0 + a \in \Theta$ for some $\gamma_0 \in \Gamma$ and $a \in K_1 \times K_2$ then for every $\gamma \in \Gamma$ we have that $\gamma + a \in \Theta$. Let $D$ be a positive integer, we thus have that $\phi_D(\gamma_0 + a) \in \Theta_D$ and so (\ref{Theta_D a union of cosets}) and the fact that $\Gamma \subset \phi_D^{-1}(\Gamma_D)$ (by definition) implies that $\Gamma_D + \phi_D(a) \subset \Theta_D$. In particular, this means that $\phi_D(\gamma + a) \in \Theta_D$. Thus there exists a $\beta_D \in \Theta$ such that $\phi_D(\gamma+a) = \phi_D(\beta_D)$. But since $\phi_D$ is a $O(2^{-D})$-almost isometry, we have that $d(\gamma + a, \beta_D) < O(2^{-D})$. Thus $\beta_D \to \gamma+a$ and since $\Theta$ is closed we must have $\gamma+a \in \Theta$. \end{proof}

Now let $U_1 \subset K_1$ and $U_2 \subset K_2$ be open sets. Let $$\mathcal{R}_i = \{n \in \bZ ~|~ P(n)\alpha_i \in U_i \}.$$

\begin{prop}\label{prop: infinite dimensional kernel contained in Stabilizer} If $\mathcal{R}_1 = \mathcal{R}_2$ and $\theta, \theta' \in \Theta$ are such that $\pi_1(\theta_1) = \pi_1(\theta_2)$, then $\pi_2(\theta - \theta') \in \operatorname{Stab}_{K_2}(\overline{U_2})$. \end{prop}

\begin{proof} We have $\theta' - \theta = (0,v)$ for some $v \in K_2$, so we wish to show $v \in  \operatorname{Stab}_{K_2}(\overline{U_2})$. By Lemma~\ref{prop: Gamma union of cosets} we have $a,a' \in A$ such that $\theta \in \Gamma + a$ and $\theta' \in \Gamma + a'$. Now let $u_2 \in U_2$. Since $\pi_2(\Gamma) = K_2$ there exists $x \in K_1$ such that $(x, u_2) \in \Gamma + a$. Now fix $\epsilon > 0$ such that the ball of radius $\epsilon$ centred at $u_2$ is contained in $U_2$. Since $\pi_1 \vert_{\Gamma}:\Gamma \to K_1$ is a surjective homomorphism between compact groups, it is an open map and thus there exists a $\delta >0$ such that if $g_1 \in K_1$ with $d(0, g_1) < \delta$ then there exists a $\gamma \in \Gamma$ with $d(0, \gamma) < \epsilon$ and $\pi_1(\gamma) = g_1$. Now since $(x, u_2) \in \Gamma +a \subset \Theta$ there exists a positive integer $n_1$ such that $d(P(n_1)\alpha, (x,u_2)) < \max\{\delta, \epsilon\}$. Writing $P(n_1)\alpha = (x_1, y_1)$ with $x_1 \in K_1, y_1 \in K_2$ we see that $d(y_1, u_2) < \epsilon$ and thus $y_1 \in U_2$ and thus since $n_1 \in \mathcal{R}_2 = \mathcal{R}_1$ we have $x_1 \in U_1$.  We set $g_1 = x_1 - x$. Notice that $d(0, g_1) < \delta$ and so there exists a $\gamma \in \Gamma$ with $d(0, \gamma) < \epsilon$ such that $\pi_1(\gamma) = g_1$. Writing $\gamma = (g_1, g_2)$ for some $g_2 \in K_2$ we get that $$(x_1, u_2 + g_2) = (x,u_2) + \gamma \in \Gamma + a + \gamma = \Gamma + a.$$ Thus $\theta'' := (x_1, u_2 + g_2 + v) \in \Gamma + a' \subset \Theta$. This means that there exists $n_2 \in \bZ$ such that $P(n_2)\alpha$ is so close to $\theta''$ that $\pi_1(P(n_2)\alpha)$ is so close to $x_1$ that $\pi_1(P(n_2)\alpha) \in U_1$ and $d(\pi_2(P(n_2)\alpha), u_2 + g_2 + v) < \epsilon$. Thus $n_2 \in \mathcal{R}_1 = \mathcal{R}_2$ and so $\pi_2(P(n_2)\alpha) \in U_2$. This means that $$d(u_2 + v, U_2) \leq d(u_2 + v + g_2, U_2) + \epsilon \leq d(u_2 + v +g_2, \pi_2(P(n_2)\alpha)) + \epsilon < 2\epsilon.$$  As $\epsilon > 0$ was arbitrary and independent of $v$, we may take $\epsilon \to 0$ to get that $d(u_2 + v, U_2) = 0$, and thus $u_2 + v \in \overline{U_2}$. Thus $U_2 + v \subset \overline{U_2}$ and by continuity of addition we have that $\overline{U_2} + v \subset \overline{U_2}$. The reverse inclusion holds by swapping the roles of $\theta$ and $\theta'$, thus $v \in \operatorname{Stab}_{K_2}(\overline{U_2})$. \end{proof}

\begin{prop} If $\mathcal{R}_1 = \mathcal{R}_2$ and $\operatorname{Stab}_{K_1}(\overline{U_1})$ and $\operatorname{Stab}_{K_2}(\overline{U_2})$ are trivial, then there exists an isomorphism of topological groups $K_1 \to K_2$ mapping $P(n)\alpha_1$ to $P(n)\alpha_2$ for all $n \in \bZ$.

\end{prop}

\begin{proof} By Proposition~\ref{prop: infinite dimensional kernel contained in Stabilizer} and the assumption $\operatorname{Stab}_{K_2}(\overline{U_2}) = \{0\}$ we have that the mapping $\pi_1\vert_{\Theta}:\Theta \to K_1$ is injective and thus a homeomorphism. But since $\pi_1(\Gamma) = K_1$ by Lemma~\ref{lemma: Gamma maps onto K_i}, we must have that $\Theta = \Gamma$, so $\pi_1\vert_{\Theta}:\Theta \to K_1$ is also a group homomorpishm.  By symmetry (this time using $\operatorname{Stab}_{K_1}(\overline{U_1}) = \{0\}$) we get that $\pi_2\vert_{\Theta}:\Theta \to K_2$ is also an isomorphism of topological groups. Thus $\pi_2\vert_{\Theta} \circ \left(\pi_1\vert_{\Theta}\right)^{-1}:K_1 \to K_2$ is the desired isomorphism. \end{proof}

\section{Return times sets for nilsystems} \label{sec: Nilsystems}

We now prove Theorem~\ref{thm: nilsystem and kronecker}. As noted, Theorem~\ref{thm: Kronecker factor of minimal system} already defines the map $\phi:X \to K$ and so it remains to provide an inverse $\psi:K \to X$, which is established by the following result.

\begin{prop} Let $X=G/\Gamma$ be a nilmanifold, where $G$ is a nilpotent Lie group and $\Gamma \leq G$ is a cocompact discrete subgroup. Suppose that $\tau \in G$ is such that $(X,T)$ is minimal, where $T:X \to X$ is given by $Tx = \tau x$ for $x \in X$. Suppose that $U \subset X$ is open and satisfies that $\operatorname{Stab}_G(\overline{U}):= \{g \in G ~|~ g\overline{U} = \overline{U}\}$ consists of only those $g \in G$ such that $gx = x$ for all $x \in X$. Now suppose that $(K,+)$ is a compact Lie group and $\alpha \in K$ is such that $\overline{\bZ \alpha} = K$ and $V \subset K$ is open such that $$\{ n \in \bZ ~|~ T^nx_0 \in U\} = \{n \in \bZ ~|~ n\alpha \in V\}$$ for some $x_0 \in X$. Then the pointed system $(X, x_0, T)$ is a factor of the pointed Kronecker system $(K,0 , k \mapsto k+\alpha)$, i.e., there is a continuous map $\psi: K \to X$ with $\psi(0) = x_0$ satisfying $$\psi(\alpha + k) = T\psi(x) \text{ for all } k \in K.$$

\end{prop}

\begin{proof} Let $Z = K \times X$ and $z_0 = (0, x_0) \in Z$. Then $Z = (K \oplus G) / (\{0\} \oplus \Gamma)$ is a nilmanifold and thus $(Z,S)$ is a nilsystem where $S:Z \to Z$ is given by $T(k,x) = (k + \alpha, Tx)$.  Now let $\Theta = \overline{S^{\bZ}z_0}$ be the closure of the $S$-orbit of $z_0$. Then (see Theorem 9 in Chapter 11 of \cite{NilpotentErgodic}) we have that $$\Theta = Hz_0$$ for some closed subgroup $H$ of $K \oplus G$. Now let $\pi_K:\Theta \to K$ and $\pi_X:\Theta \to X$ denote the projection maps. We wish to show that $\pi_K$ is injective as it would then be a homeomorphism and our desired factor map would be $\pi_X \circ \pi_K^{-1}:K \to X$. Thus suppose that $\theta_1, \theta_2 \in \Theta$ are such that $\pi_K(\theta_1) = \pi_K(\theta_2)$. Thus we can write $\theta_1 = h_1 z_0$ and $\theta_2 = h_2 z_0$ for some $h_1,h_2 \in H$. For $i=1,2$, we may write $h_i = (k_i, g_i)$ where $k_i \in K$ and $g_i \in G$, thus $\theta_i = (k_i, g_i x_0)$. But $k_1 = k_2$ and so we have $h_0 := h_2h_1^{-1} = (0, g_2 g_1^{-1}) \in H$ and we have the relation $h_0 \theta_1 = \theta_2$. Now let $u \in U$ be arbitrary. By surjectivity of $\pi_X:\Theta \to X$ we may find a $h' \in H$ such that $h'z_0 = (k', u) \in \Theta$ for some $k' \in K$. Thus we may find a sequence of integers $n_1, n_2, \ldots$ such that $S^{n_i}z_0 \to h'z_0$. In particular $T^{n_i}x_0 \to u$ and thus $T^{n_i}x_0 \in U$ for large enough $n_i$. This must mean that $n_i \alpha \in V$. Now $h_0 S^{n_i}z_0 = (n_i\alpha, g_2g_1^{-1}T^{n_i}z_0) \in \Theta$ and so, for each fixed $i$, we may find a sequence of integers $m_1, m_2, \ldots $ such that $\lim_{j \to \infty} S^{m_j}z_0 = (n_i\alpha, g_2g_1^{-1}T^{n_i}x_0)$. In particular this means that $m_j\alpha \to n_i\alpha \in V$ and so for large enough $j$ we have $m_j \alpha \in V$ and so $T^{m_j}x_0 \in U$.  But $T^{m_j}x_0 \to g_2g_1^{-1}T^{n_i}x_0$ and so $g_2g_1^{-1}T^{n_i}x_0 \in \overline{U}$. Finally taking the limit as $i \to \infty$ and using $T^{n_i}x_0 \to u$ we have that $g_2g_1^{-1} u\in \overline{U}$. As $u \in U$ was arbitrary, this shows that $g_2g_1^{-1} U \subset \overline{U}$ and thus $g_2g_1^{-1} \overline{U} \subset \overline{U}$. The reverse inclusion follows from swapping the roles of $\theta_1$ and $\theta_2$, thus $g_2g_1^{-1} \in \operatorname{Stab}_G(\overline{U})$. Thus $g_2g_1^{-1}$ stabilizes every $x \in X$ thus $g_2x_0 = g_2g_1^{-1}g_1x_0 = g_1x_0$, so in fact $\theta_2 = \theta_1$. Thus $\pi_K$ is injective as desired. \end{proof}

\section{Spectral construction for Jordan Measurable sets (Proof of Theorem~\ref{thm: explicit reconstruction})}

Let $U \subset K$ be a Jordan measurable open subset of a compact abelian group $K$. Let $\alpha \in K$ be an element such that $\bZ \alpha$ is dense in $K$. Let $\Lambda$ denote the set of unit complex numbers $\lambda$ such that

$$\frac{1}{N}\sum_{n=1}^N \lambda^n \mathds{1}_{U}(n\alpha)$$ does not converge to $0$ as $N \to \infty$.

\begin{lemma} We have $\Lambda \subset \{ \chi(\alpha) ~|~ \chi \in \widehat{K} \}.$ In fact, if $$\mathds{1}_U = \sum_{\chi} c_{\chi} \chi$$ is the Fourier decomposition in $L^2(K)$, then $\Lambda$ is the set of those $\chi(\alpha)$ for which $c_{\chi} \neq 0$. 

\end{lemma}

\begin{proof} Let $\bU$ be the unit complex numbers. Suppose that $\lambda \in \bU$ but $\lambda \notin  \{ \chi(\alpha) ~|~ \chi \in \widehat{K} \}$. Note that for any character $\chi$ on $K$ we have that $$\frac{1}{N}\sum_{n=1}^N \lambda^n \chi(n\alpha) = \frac{1}{N}\sum_{n=1}^N(\lambda\chi(\alpha))^n \to 0.$$ As any continuous function $f:K \to \bC$ can be uniformly approximated by a linear combination of characters, we have that $$\frac{1}{N}\sum_{n=1}^N \lambda^n f(n\alpha) \to 0.$$

Now observe that there exist continuous functions $f_{\epsilon}^+: K \to [0,1]$ such that $$f_{\epsilon}^+ \searrow \mathds{1}_{\overline{U}}, \quad \text{ as } \epsilon \to 0.$$ In particular we choose these functions so that $$m_K(\operatorname{supp}f_{\epsilon}^+ \setminus U) = m_K(\operatorname{supp}f_{\epsilon}^+ \setminus \overline{U}) < \epsilon$$ where $m_K$ denotes the Haar measure on $K$ and in the first equality we used that $\overline{U}$ and $U$ are $m_K$ almost the same (as $U$ is Jordan measurable). Now as $\operatorname{supp}f_{\epsilon}^+ \setminus U$ is a closed set its indicator function can be written as a pointwise decreasing limit of continuous functions, thus we have that $$\limsup_{N \to \infty} \frac{1}{N} |\{ n \in [1,N] ~|~ n\alpha \in \operatorname{supp}f_{\epsilon}^+ \setminus U \}| \leq m_K(\operatorname{supp}f_{\epsilon}^+ \setminus U) < \epsilon.$$ This means that $$\limsup_{N \to \infty} |\frac{1}{N} \sum_{n=1}^N \lambda^n f_{\epsilon}^+ (n\alpha) - \frac{1}{N} \sum_{n=1}^N \lambda^n \mathds{1}_{U}(n\alpha)| < \epsilon.$$ Thus $$\frac{1}{N} \sum_{n=1}^N \lambda^n \mathds{1}_{U}(n\alpha) \to 0$$ and hence we have shown the first claim that $$\Lambda \subset \{ \chi(\alpha) ~|~ \chi \in \widehat{K} \}.$$

Now suppose that $\lambda \in \{ \chi(\alpha) ~|~ \chi \in \widehat{K} \}$. Thus $\lambda = \overline{\chi(\alpha)}$ for some unique (by density of $\bZ\alpha$) character $\chi$ on $K$. But $\overline{\chi} \mathds{1}_U$ is Riemann-integrable thus we have that $$\frac{1}{N} \sum_{n=1}^N \lambda^n\mathds{1}_U(n\alpha)  = \frac{1}{N} \sum_{n=1}^N \overline{\chi}\mathds{1}_U(n\alpha) \to \int \overline{\chi}\mathds{1}_U dm_K = c_{\chi}.$$ 
\end{proof}

Now let $E = \{ \chi(\alpha) ~|~ \chi \in \widehat{K} \}$. Note that the mapping $\widehat{K} \to E$ mapping $\chi$ to $\chi(\alpha)$ is injective (by density of $\bZ \alpha$) and thus bijective. Consequently, let $\Gamma \subset \widehat{K}$ be the set corresponding to $\Lambda$, i.e., $\Lambda = \{ \chi(\alpha) ~|~ \chi \in \Gamma\}$. 

Now observe that the mapping $\iota:K \to \bU ^{\widehat{K}}$ given by $\iota(k) = \left( \chi(k) \right)_{\chi \in \widehat{K}}$ is injective and continuous, thus a homeomorphism onto its image. Now let $$\pi: K \to \bU^{\Gamma}$$ be the projection given by $$k \mapsto \left( \chi(k) \right)_{\chi \in \Gamma}.$$

\begin{prop} \label{prop: kernel inside stabilizer, spectral} $\operatorname{ker}(\pi) \subset \operatorname{Stab}_K(\overline{U})$.

\end{prop} 

\begin{proof} Suppose that $k \in \operatorname{ker} \pi$. Thus $\chi(k) = 1$ for all $\chi \in \Gamma$. But since $$\mathds{1}_{\overline{U}} = \sum_{\chi \in \Gamma} c_{\chi} \chi$$ we have that $\mathds{1}_{\overline{U}}(x + k ) = \mathds{1}_{\overline{U}}(x)$ for almost all $x \in K$. As $U$ and $\overline{U}$ are $m_K$ almost equal, this means that the open set $(U - k) \setminus \overline{U}$ has zero measure and is thus empty. This means that $u - k \in \overline{U}$ for all $u \in U$ and $k \in \operatorname{ker}(\pi)$. Thus $\operatorname{ker(\pi)} \subset \operatorname{Stab}_K(\overline{U})$. \end{proof}

\begin{proof}[Proof of Theorem~\ref{thm: explicit reconstruction}]  The assumption that $\operatorname{Stab}_K(\overline{U})$ is trivial together with Proposition~\ref{prop: kernel inside stabilizer, spectral} implies that $\pi$ is injective and thus an isomorphism onto its image. But the image of $\pi$ is the closed subgroup of $\bU^{\Gamma}$ generated by $$\pi(\alpha) =  \left( \chi(\alpha) \right)_{\chi \in \Gamma}.$$ Finally, by applying the natural isomorphism $\bU^{\Gamma} \cong \bU^{\Lambda}$ (i.e., the one induced by the bijection $\chi \mapsto \chi(\alpha)$ from $\Gamma$ to $\Lambda$) we get that $K$ is isomorphic to the closed subgroup of $\bU^{\Lambda}$ generated by $ \left( \lambda \right)_{\lambda \in \Lambda}. $ \end{proof}

\section{Equivalence of Question~\ref{question: G sets} and Question~\ref{question: stable implies simple}}
\label{section: equivalence of questions}

We now show that Question~\ref{question: G sets} and Question~\ref{question: stable implies simple} are equivalent. Recall that a \textit{block system} of a group action $G\curvearrowright X$ is a partition $\mathcal{B} \subset 2^X$ of $X$ (collection of non-empty disjoint sets whose union is $X$) that is $G$-invariant (i.e., if $g \in G$ and $B \in \mathcal{B}$ then $gB \in \mathcal{B}$). Let us define the \textit{block system generated by $U \subset X$} (for this action) to be the smallest (with respect to inclusion) block system such that $U$ is a union of elements of $\mathcal{B}$. More concretely, the the block system generated by $U$ consists of those minimal non-empty sets that can be written as intersections of elements in $\{ gU ~|~ g \in G\} \cup \{ X \setminus gU ~|~ g \in G\}$. 

\begin{prop}\label{prop: return times to simple sets} Let $G$ be a group acting transitively on a set $X_1$ and also acting transitively on a set $X_2$ and suppose that $U_1 \subset X_1$ and $U_2 \subset X_2$ are simple for the respective actions. Suppose that $x_1 \in X_1$ and $x_2 \in X_2$ are such that $$\{g \in G ~|~ gx_1 \in U_1\} = \{g \in G ~|~ gx_2 \in U_2\}.$$ Then there is an isomorphism $\phi:X_1 \to X_2$ of $G$-sets mapping $x_1$ to $x_2$, i.e., $\phi(gx) = g\phi(x)$ for all $g \in G$ and $x \in X_1$.

\end{prop}

\begin{proof} Let $\mathcal{B}$ denote the block system generated by $U_1$. Thus there is a well defined map $\pi: X_1 \to \mathcal{B}$ where for $x \in X_1$ we define $\pi(x)$ to be the unique element of $\mathcal{B}$ containing $x \in X_1$. Note that $\pi$ is a morphism of $G$-actions, i.e., $g \pi(x) = \pi(gx)$ for all $g \in G$ and $x \in X_1$. Note that $$U_1 = \pi^{-1}\left( \bigcup_{\{B \in \mathcal{B} ~|~B \subset U_1\}} B \right)$$ since by definition $U_1$ is a union of elements of $\mathcal{B}$. Thus since $U_1$ is simple we must have that $\pi$ is an isomorphism. In particular this means $\{x\} \in \mathcal{B}$ for all $x \in X_1$. Now let $$R(V) = \{g \in G ~|~ gx_1 \in V\}$$ for any $V \subset X_1$. Notice the properties $R(gV)=gR(V)$ and $R(\bigcap_{i \in \mathcal{I}} V_i)= \bigcap_{i \in \mathcal{I}} R(V_i)$ and $R(X_1 \setminus V) = G \setminus R(V)$. Now let $\mathcal{B}_G$ denote the block system generated by $R(U_1)$ for the action of $G \curvearrowright G$ by left multiplication. By the aforementioned properties we have that $\mathcal{B}_G$ consists of those sets of the form $R(B)$ for $B \in \mathcal{B}$ and so we have that $\mathcal{B}_G$ consists of exactly the sets of the form $R(\{x\})$ for $x \in X_1$. Note that $R(\{x_1\}) = \{g \in G ~|~ gx_1 = x_1\} = \operatorname{Stab}(x_1)$ and any other $R(\{x\})$ must be a coset of $\operatorname{Stab}(x_1)$ (by transitivity $x = gx_1$ for some $g \in G$ and so $R(\{x\}) = gR(\{x_1\})$). So in fact $\operatorname{Stab}(x_1)$ may be desribed as member of $\mathcal{B}_G$ that contains $1 \in G$. Since $X_1 \cong G/\operatorname{Stab}(x_1)$, this means that we have an isomorphism of $G$-actions $X_1 \to \mathcal{B}_G$ that maps $x_1$ to the unique element of $\mathcal{B}_G$ containing $1 \in G$ (it maps $x$ to $R(\{x\})$). Notice that $\mathcal{B}_G$ is the same if we replace $X_1,x_1,U_1$ with $X_2, x_2, U_2$ respectively by the assumption that $$\{g \in G ~|~ gx_1 \in U_1\} = \{g \in G ~|~ gx_2 \in U_2\}.$$ Thus indeed we have an isomorphism $X_1 \cong X_2$ of $G$-actions (they are both isomorphic to $G$ acting on $\mathcal{B}_G$). \end{proof}

This demonstrates the equivalence of the two questions as follows. If the answer to Question~\ref{question: stable implies simple} is affirmative, then the two sets in Question~\ref{question: G sets} are simple and thus the Proposition~\ref{prop: return times to simple sets} provides an affirmative answer. Conversely, if $U \subset X$ is a subset in a transitive $G$-set $X$ with a trivial setwise stabilizer, then we have a factor $\pi:X \to \mathcal{B}$, where $\mathcal{B}$ is the block system generated by $U$, mapping $x \in X$ to the unique element of $\mathcal{B}$ containing $x$ (as constructed in the proof of the Propostion~\ref{prop: return times to simple sets}). But $\pi(U)$ also has a trivial setwise stabilizer and $\{g \in G ~|~ gx_0 \in U\} = \{g \in G ~|~ g\pi(x_0) \in \pi(U)\}$ so if the answer to Question~\ref{question: G sets} is affirmative, then $\pi$ must be an isomorphism and so $\mathcal{B}$ consists of singletons.

\end{document}